\documentclass[12pt]{amsart}
\usepackage{amsmath,amssymb,amsthm,amscd,verbatim}
\usepackage[dvips]{graphicx}
\theoremstyle{plain}
\newtheorem{thm}{Theorem}[section]
\newtheorem{lem}[thm]{Lemma}
\newtheorem{cor}[thm]{Corollary}

\theoremstyle{definition}

\theoremstyle{remark}
\newtheorem{rem}[thm]{Remark}

\begin{document}
\baselineskip=19pt
\title{On diffeomorphisms over non-orientable surfaces \\ 
standardly embedded in the 4-sphere}
\author{Susumu Hirose}
\address{Department of Mathematics,  
Faculty of Science and Technology, 
Tokyo University of Science, 
Noda, Chiba, 278-8510, Japan} 
\email{hirose\b{ }susumu@ma.noda.tus.ac.jp}
\thanks{This research was supported by Grant-in-Aid for 
Scientific Research (C) (No. 20540083), 
Japan Society for the Promotion of Science. }
\begin{abstract} 
For a non-orientable closed surface standardly embedded in the 4-sphere, 
a diffeomorphism over this surface is extendable 
if and only if this diffeomorphism preserves 
the Guillou-Marin quadratic form of this embedded surface. 
\end{abstract}
\maketitle

\section{Introduction}

Let $S$ be a closed surface and $e$ be a smooth embedding of $S$ into $S^4$. 
A diffeomorphism $\phi$ over $S$ is {\em $e$-extendable\/} if 
there is an orientation preserving diffeomorphism $\Phi$ of $S^4$ 
such that $\Phi |_{S} = \phi$. 
The natural problem to ask is: 
{\em Find a necessary and sufficient condition 
for a diffeomorphism over $S$ to be $e$-extendable?\/} 

For some special embeddings of closed surfaces in $4$-manifolds, 
we have answers to the above problem (for example, \cite{Montesinos}, 
\cite{Hirose}, \cite{H-Y}). 
An embedding $e$ of the orientable surface $\Sigma_g$ into $S^4$ is called standard 
if $e(\Sigma_g)$ is the boundary of $3$-dimensional handlebody embedded in $S^4$. 
In \cite{Montesinos} and \cite{Hirose}, we showed: 
\begin{thm}[\cite{Montesinos} ($g=1$), \cite{Hirose} ($g \geq 2$)] 
Let $\Sigma_g$ be standardly embedded in $S^4$. 
An orientation preserving diffeomorphism $\phi$ over 
the $\Sigma_g$ is extendable to $S^4$ 
if and only if $\phi$ preserves the Rokhlin quadratic form 
of the $\Sigma_g$ standardly embedded in $S^4$. 
\end{thm}
In this paper, we consider the same kind of problem for 
non-orientable surfaces embedded in $S^4$. 
\begin{figure}[hbtp]
\includegraphics[height=15cm]{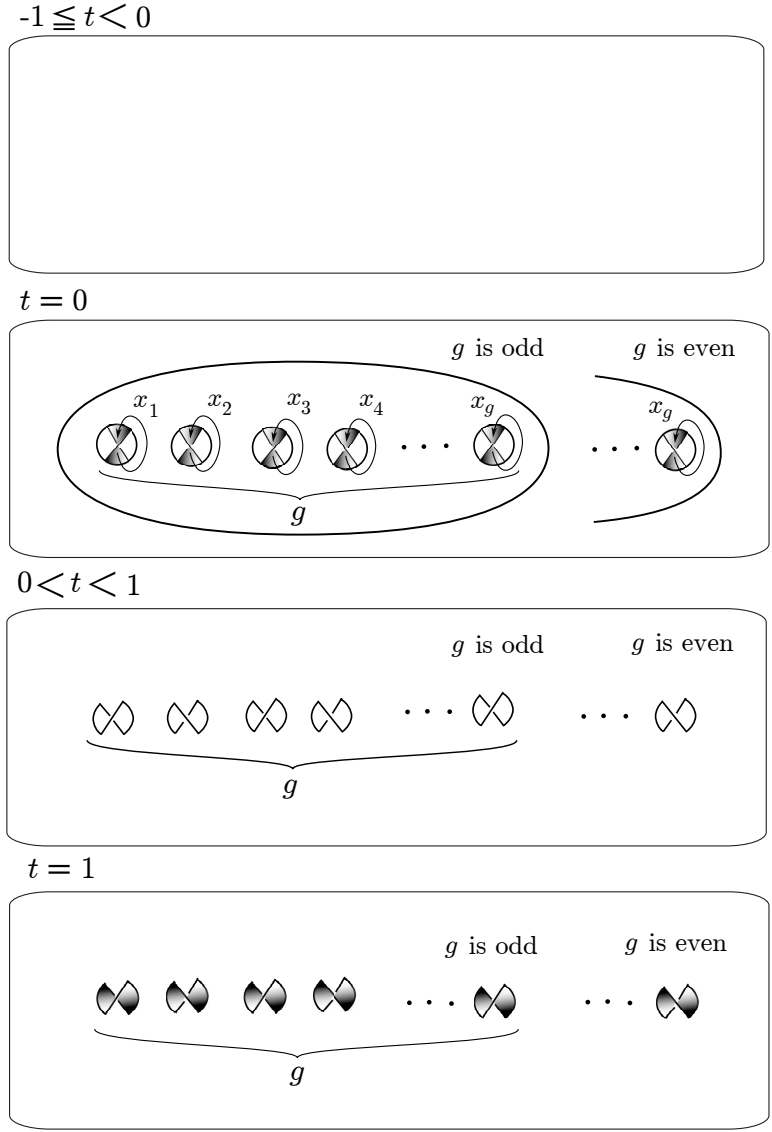}
\caption{
%
The motion picture of the $o$-standard embedding of $N_g$ into $S^4$.}
\label{fig:o-standard-basis}
\end{figure}
Let $N_g$ be a connected non-orientable surface constructed 
from $g$ projective planes by connected sum. 
We call $N_g$ {\em the closed non-orientable surface of genus\/} $g$. 
Let $S^3 \times [-1,1]$ be a closed tubular neighborhood of the equator $S^3$ 
in $S^4$. Then $S^4 - S^3 \times(-1,1)$ consists of two $4$-balls. 
An embedding $os : N_g \hookrightarrow S^4$ is {\em $o$-standard\/} 
if $os(N_g) \subset S^3 \times [-1,1]$ and as shown in Figure \ref{fig:o-standard-basis}.
The main result of this paper is:
\begin{thm}\label{thm:extendable} 
%
The diffeomorphism $\phi$ over $N_g$ is $os$-extendable 
if and only if $\phi$ preserves the Guillou-Marin quadratic form of the $N_g$ 
$o$-standardly embedded in $S^4$.  
\end{thm}
\section{Guillou-Marin quadratic form}\label{sec:G-M form}

For a smooth embedding $e$ of 
the closed non-orientable surface $N_g$ of genus $g$ into $S^4$, 
Guillou and Marin (\cite{G-M} see also \cite{Matsumoto}) defined 
a quadratic form $q_e : H_1(N_g ; \mathbb{Z}_2) \to \mathbb{Z}_4$ 
as follows: 
Let $C$ be an immersed circle on $N_g$, and 
$D$ be a connected orientable surface immersed in $S^4$ 
such that $\partial D = C$, and $D$ is not tangent to $N_g$. 
Let $\nu_D$ be the normal bundle of $D$, then 
$\nu_D |_C $ is a solid torus with a trivialization induced from 
any trivialization of $\nu_D$.  
Let $N_{N_g}(C)$ be the tubular neighborhood of $C$ in $N_g$, then 
$N_{N_g}(C)$ is an twisted annulus or a M\"obius band in $\nu_D |_C $. 
We denote by $n(D)$ the number of right hand half-twists of $N_{N_g}(C)$ 
with respect to the trivialization of $\nu_D |_C$. 
Let $D \cdot F$ be the mod-$2$ intersection number between $D$ and $F$, 
$Self(C)$ be the mod-$2$ double points number of $C$, 
and $2\times$ be an injection $\mathbb{Z}_2 \to \mathbb{Z}_4$ 
defined by $2 \times [n]_2 = [2n]_4$. 
Then the number $n(D) + 2 \times D \cdot F + 2 \times Self(C)\pmod 4 $ depend only 
on the mod-$2$ homology class $[C]$ of $C$. 
Hence, we define 
$$q_e ([C]) = n(D) + 2 \times D \cdot F + 2 \times Self(C) \pmod 4. $$
This map $q_e$ is called Guillou-Marin {\em quadratic form\/}, 
since $q_e$ satisfies 
$$q_{e} (x+y) = q_{e}(x) + q_{e}(y) + 2 \times (x \cdot y)_2, $$ 
where $(x \cdot y)_2$ means the mod-$2$ intersection number between $x$ and $y$. 
For example, $q_{os}(x_{2i-1}) = +1, q_{os}(x_{2i}) = -1$ 
for the basis $\{ x_1, \ldots, x_g \}$ 
of $H_1(N_g ; \mathbb{Z}_2)$ shown in Figure \ref{fig:o-standard-basis}. 
This quadratic form $q_e$ is a non-orientable analogy of 
Rokhlin quadratic form. 

A diffeomorphism $\phi$ over $N_g$ is {\em $e$-extendable\/} 
if there is an orientation preserving diffeomorphism $\Phi$ of $S^4$ 
such that the following diagram is commutative, 
\begin{equation*} 
\begin{CD}
N_g @>{e}>> S^4 \\
@V{\phi}VV @VV{\Phi}V \\
N_g @>{e}>> S^4. 
\end{CD}
\end{equation*}
If the diffeomorphisms $\phi_1$ over $N_g$ is $e$-extendable, 
and $\phi_1$ is isotopic to $\phi_2$, 
then $\phi_2$ is $e$-extendable. 
Therefore, $e$-extendability is a property about isotopy classes of 
diffeomorphisms over $N_g$. 
The group $\mathcal{M}(N_g)$ of isotopy classes of all diffeomorphisms 
over $N_g$ is called {\em the mapping class group of $N_g$\/}. 
An element $\phi$ of $\mathcal{M}(N_g)$ is {\em $e$-extendable\/} 
if there is an $e$-extendable representative of $\phi$.   
By the definition of $q_e$, we can see that 
if $\phi \in \mathcal{M}(N_g)$ is $e$-extendable then $\phi$ preserves $q_e$, i.e. 
$q_e ( \phi_*(x)) = q_e (x)$ for every $x \in H_1(N_g ; \mathbb{Z}_2)$. 
What we would like to know is whether $\phi \in \mathcal{M}(N_g)$ is $e$-extendable 
when $\phi$ preserves $q_e$. 
The answer to this problem would be depend on the embedding $e$. 
In this paper, we consider the case where $e$ is the $o$-standard embedding. 

\section{Generators for $\mathcal{M}(N_g)$}
\begin{figure}[hbtp]
\includegraphics[height=4cm]{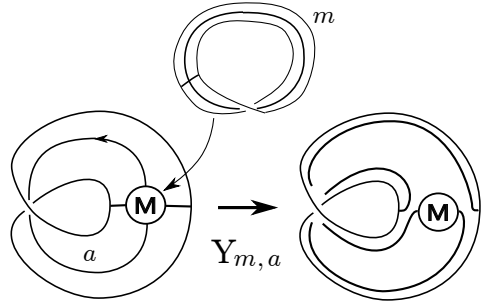}
\caption{$M$ with circle indicates a place where to attach a M\"obius band}
\label{fig:Y-homeo}
\end{figure}
A simple closed curve $c$ on $N_g$ is an A-circle (resp. an M-circle), 
if the tubular neighborhood of $c$ is an annulus (resp. a M\"obius band). 
We denote by $t_c$ the Dehn twist about an A-circle $c$ on $N_g$. 
In each figure, we indicate the direction of a Dehn twist by an arrow. 
Lickorish \cite{Lickorish1, Lickorish2} showed that Dehn twists 
and {\it $Y$-homeomorphisms\/} generate $\mathcal{M}(N_g)$. 
We review the definition of $Y$-homeomorphism. 
Let $m$ be an M-circle 
and $a$ be an oriented A-circle in $N_g$ 
such that $m$ and $a$ transversely intersect in one point. 
Let $K \subset N_g$ be a regular neighborhood of $m \cup a$, 
which is a unioun of the tubular neighborhoods of $m$ and $a$ 
and then is homeomorphic to the Klein bottle with a hole. 
Let $M$ be a regular neighborhood of $m$. 
We denote by $Y_{m,a}$ a homeomorphism over $N_g$ which is described 
as the result of pushing $M$ once along $a$ keeping the boundary of $K$ fixed 
(see Figure \ref{fig:Y-homeo}). 
We call $Y_{m,a}$ a $Y$-homeomorphism. 
Szepietowski \cite{Szepietowski} showed an interesting 
results on the proper subgroup of $\mathcal{M}(N_g)$ generated 
by all $Y$-homeomorphisms. 

\begin{thm}[\cite{Szepietowski}]\label{thm:level2}
%
$\Gamma_2(N_g) = \{ \phi \in \mathcal{M}(N_g) \, | 
\, \phi_* : H_1(N_g ; \mathbb{Z}_2) \to H_1(N_g ; \mathbb{Z}_2) = id \}$ 
is generated by $Y$-homeomorphisms. 
\end{thm}

Chillingworth showed that $\mathcal{M}(N_g)$ is finitely generated. 
\begin{figure}[hbtp]
\includegraphics[height=6.5cm]{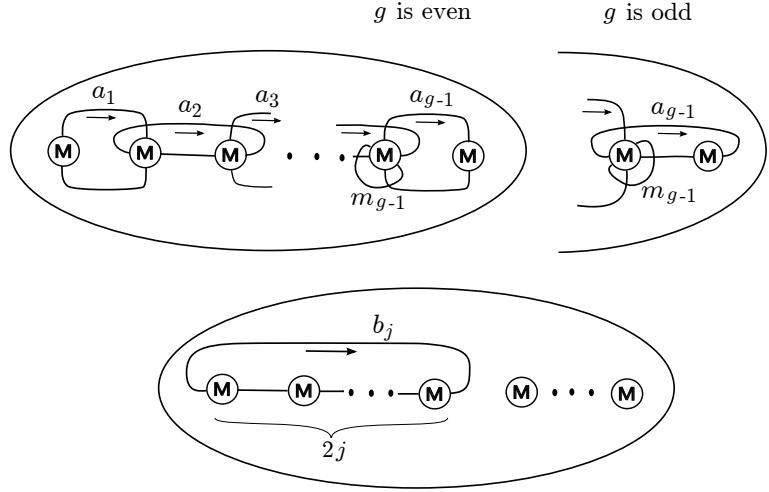}
\caption{Chillingworth's generators for $\mathcal{M}(N_g)$. }
\label{fig:generator}
\end{figure}
\begin{thm}[\cite{Chillingworth}]
\label{thm:Chillingworth}
Let $a_1, \ldots, a_{g-1}$, $b_j$ $(1 \leq j \leq 2/g)$ and $m_{g-1}$ 
be circles shown in Figure \ref{fig:generator}. 
Then $t_{a_1}, \ldots, t_{a_{g-1}}, t_{b_j}\ (1 \leq j \leq 2/g)$, 
$Y_{m_{g-1}, a_{g-1}}$ 
generate $\mathcal{M}(N_g)$. 
\end{thm}
\begin{rem}
If $g=1$, then $\mathcal{M}(N_1)$ is trivial, hence 
Theorem \ref{thm:extendable} is valid. 
From here to the end of this paper, we assume $g \geq 2$. 
\end{rem}
\begin{figure}[hbtp]
\includegraphics[height=6cm]{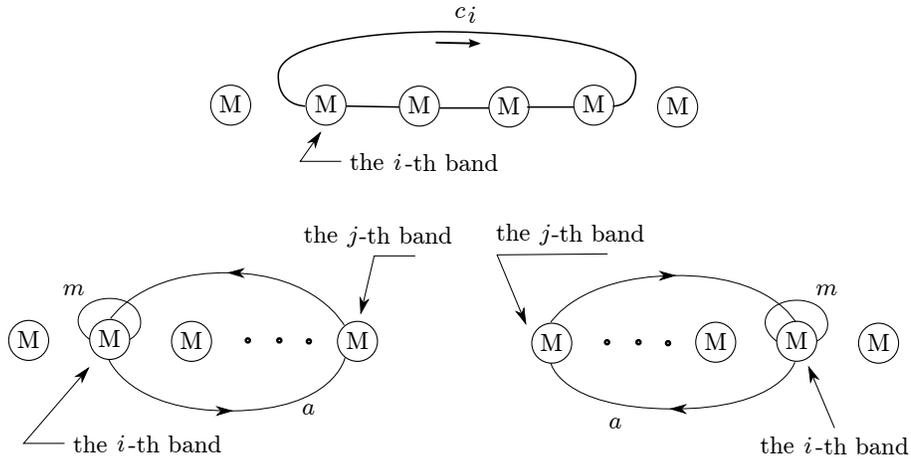}
\caption{
%
The circle $c_i$ and the Y-homeomorphism $Y_{i,j}$.}
\label{fig:generator2}
\end{figure}
{\em The $i$-th band\/} is the M\"{o}bius band on which the circle $x_i$ 
in Figure \ref{fig:o-standard-basis} goes across.  
Let $c_i$ ($ i = 1, \ldots, g-3$) be a simple closed curve shown in 
the top of Figure \ref{fig:generator2}. 
Let $i, j = 1, \ldots, g$ such that $i \not= j$. 
When $i < j$ (resp. $i > j$), we define 
$Y_{i,j} = Y_{m,a}$, where $m$ and $a$ are as shown 
in the left bottom (resp. the right bottom) of Figure \ref{fig:generator2}. 
Let $\mathcal{YS}_g$ be the subgroup of $\mathcal{M}(N_g)$ generated by 
all $Y_{i,j}$. 
\begin{lem}
\label{lem:generator} 
%
$\mathcal{YS}_g$ and 
$t_{a_1}, \ldots, t_{a_{g-1}}$, $t_{c_1}, \ldots, t_{c_{g-3}}$ 
generate $\mathcal{M}(N_g)$.
\end{lem}
\begin{proof}
\begin{figure}[hbtp]
\includegraphics[height=6cm]{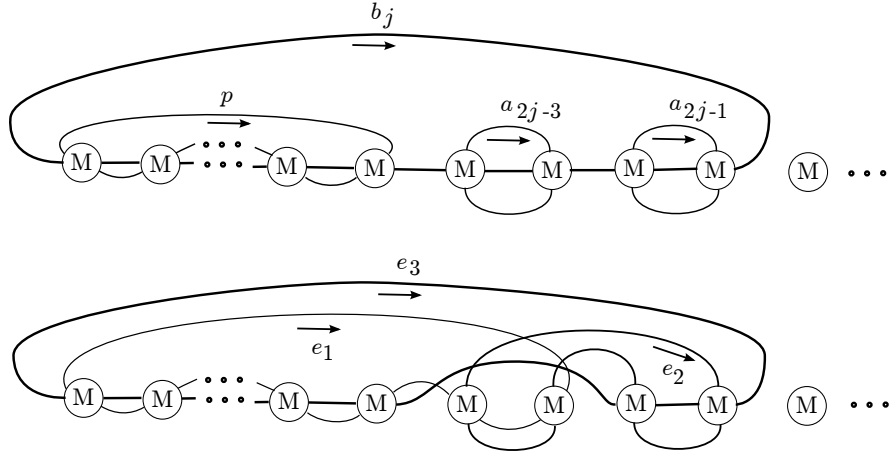}
\caption{
%
When $j=3$, we drop the 2-nd and 3-rd $M$ with circles. 
}
\label{fig:lantern}
\end{figure}
It suffices to show that $t_{b_j}$ is a product of 
$t_{a_1}, \ldots, t_{a_{g-1}}$, $t_{c_1}, \ldots, t_{c_{g-3}}$. 
When $j=2$, $t_{b_2} = t_{c_1}$. 
When $j\geq 3$, by the lantern relation 
which was discovered by Dehn and rediscovered by Johnson 
\cite{Johnson}, 
$t_{e_1} t_{e_2} t_{e_3} = t_p t_{a_{2j-3}} t_{a_{2j-1}} t_{b_j}$, 
where $p$, $e_1, e_2$ and $e_3$ are circles shown in Figure \ref{fig:lantern}, 
hence 
$t_{b_j} = (t_p t_{a_{2j-3}} t_{a_{2j-1}})^{-1} t_{e_1} t_{e_2} t_{e_3}$. 
Since $e_1 = t_{a_{2j-2}} t_{a_{2j-3}} t_{a_{2j-1}} t_{a_{2j-2}} (e_3)$, 
we see 
$t_{e_3} = (t_{a_{2j-2}} t_{a_{2j-3}} t_{a_{2j-1}} t_{a_{2j-2}})^{-1} 
t_{e_1} t_{a_{2j-2}} t_{a_{2j-3}} t_{a_{2j-1}} t_{a_{2j-2}}$. 
If $j=3$, then $e_1 = c_1$, $e_2 = c_3$ and $p = a_1$. 
Therefore, $t_{b_3}$ is a product of 
$t_{a_1}, \ldots, t_{a_{g-1}}$, $t_{c_1}, \ldots, t_{c_{g-3}}$. 
If $j \geq 4$, then $e_1 = b_{j-1}$, $e_2 = c_{2j-3}$, and $p = b_{j-2}$. 
By the induction on $j$, we see that 
$t_{b_j}$ is a product of 
$t_{a_1}, \ldots, t_{a_{g-1}}$, $t_{c_1}, \ldots, t_{c_{g-3}}$. 
\end{proof}
\begin{rem}
If $g=2,3$, then this lemma reads $\mathcal{YS}_g$ and 
$t_{a_1}, \ldots, t_{a_{g-1}}$ generate $\mathcal{M}(N_g)$.   
\end{rem}
\section{Generators for subgroup of $\mathcal{M}(N_g)$ preserving 
$q_{os}$
}
In this section, we find a finite system of generators for  
$$
\mathcal{N}_g (q_{os})= 
\left\{ \phi \in \mathcal{M}(N_g) \, \left| \, 
q_{os} ( \phi_* (x)) = q_{os} (x) 
\text{ for every } x \in H_1(N_g;\mathbb{Z}_2)
\right.
\right\} 
$$
and prove the main theorem (Theorem \ref{thm:extendable}) of 
this paper. \par
We introduce a group 
$$
\mathcal{O}_g (q_{os})= 
\left\{ A \in Aut(H_1(N_g; \mathbb{Z}_2)) \, | \, 
q_{os} ( A (x)) = q_{os} (x) 
\text{ for every } x \in H_1(N_g;\mathbb{Z}_2)
\right\}. 
$$
Then we have a natural short exact sequence 
%
\begin{equation}\label{eq:short-exact}
0 \to \Gamma_2(N_g) \to \mathcal{N}_g (q_{os}) 
\to \mathcal{O}_g (q_{os}) \to 0. 
\end{equation}
Since $\Gamma_2(N_g)$ is a finite index subgroup of $\mathcal{M}(N_g)$ and 
$\mathcal{O}_g (q_{os})$ is a finite group, 
there exists a finite system of generators for $\mathcal{N}_g (q_{os})$. 
We find a system of generators explicitly. 
\begin{thm}\label{thm:generator-pin}
%
\begin{figure}[hbtp]
\includegraphics[height=3cm]{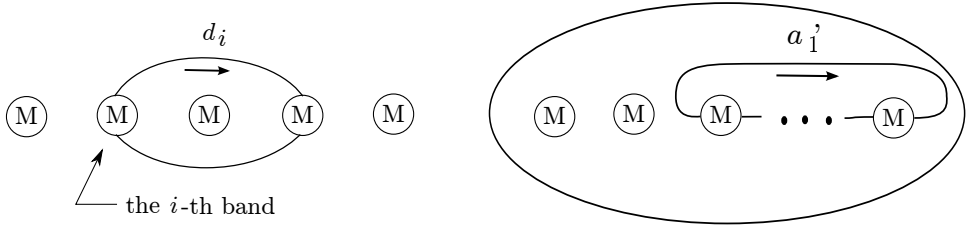}
\caption{
%
The circle $d_i$. 
}
\label{fig:generator3}
\end{figure}
$\mathcal{N}_g (q_{os})$ is generated by 
$\mathcal{YS}_g$, 
$t_{a_1}^2$, $\ldots$, $t_{a_{g-1}}^2$, 
$t_{c_1}^2$, $\ldots$, $t_{c_{g-3}}^2$, 
$t_{d_1}$, $\ldots$, $t_{d_{g-2}}$, 
and $t_{a_1} t_{a_3} t_{c_1}$, $\ldots$, 
$t_{a_{g-3}} t_{a_{g-1}} t_{c_{g-3}}$, 
where $d_i$ is illustrated 
in Figure \ref{fig:generator3}.  
\end{thm}
\begin{rem}
If $g=2$, this theorem reads $\mathcal{N}_2(q_{os})$ is generated by 
$\mathcal{YS}_2$ and $t_{a_1}^2 = id_{N_2}$. 
If $g=3$, this theorem reads $\mathcal{N}_3(q_{os})$ is generated by 
$\mathcal{YS}_3$, $t_{a_1}^2, t_{a_2}^2$ and $t_{d_1}$.  
\end{rem}
{\it Proof of Theorem \ref{thm:extendable}.} 
By the definition of $q_{os}$, if a diffeomorphism $\phi$ over $N_g$ is $os$-extendable 
then $\phi$ preserves $q_{os}$. \par
\begin{figure}[hbtp]
\includegraphics[height=2.5cm]{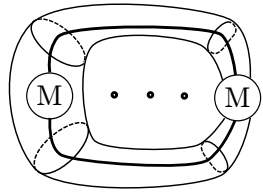}
\caption{
Sliding the left M\"obius band along this tube is an extension of $Y_{i,j}$.}
\label{fig:Y-tube}
\end{figure}
Conversely, we assume that $\phi$ preserves $q_{os}$. 
Then $\phi$ is an element of $\mathcal{N}_g(q_{os})$. 
Therefore, if each generator of $\mathcal{N}_g(q_{os})$ is $os$-extendable then 
$\phi$ is $os$-extendable. 
Since a sliding of a M\"obius band along 
the tube illustrated in Figure \ref{fig:Y-tube} is an extension of $Y_{i, j}$, 
$Y_{i,j}$ is $os$-extendable, hence every element of $\mathcal{YS}_g$ is 
$os$-extendable. 
Since the regular neighborhoods of $a_i$ and $c_i$ are annuli 
trivially embedded in the equator $S^3$ of $S^4$, 
$t_{a_i}^2$, $t_{c_i}^2$ are $os$-extendable by 
the same argument as in the introduction of \cite{Hirose2}. 
Since the regular neighborhoods of $d_i$ is a Hopf band embedded 
in the equator $S^3$ of $S^4$, $t_{d_i}$ are $os$-extendable by 
the same argument as the proof of \cite[Proposition 2.1]{Hirose2}. 
Finally, by using the same argument as showing the extendability of 
``$C_1 C_3 C_5$" in the proof of \cite[Lemma 2.2]{Hirose}, 
we show that $t_{a_i} t_{a_{i+2}} t_{c_i}$ is $os$-extendable. 
\qed \par
For $a \in H_1(N_g; \mathbb{Z}_2)$, 
we define {\em the transvection\/} 
$T_a : H_1(N_g;\mathbb{Z}_2) \to H_1(N_g;\mathbb{Z}_2)$ 
{\em about $a$\/} by $T_a(x) = x + (x \cdot a)_2 \ a,$ where $(\ \cdot \ )_2$ 
means the mod-$2$ intersection form. 
We remark that if $l$ is a simple closed curve on $N_g$ such that 
$[l] = a \in H_1(N_g; \mathbb{Z}_2)$, then $(t_l)_* = T_a$. 
Since ${T_a}^2 = id_{H_1(N_g; \mathbb{Z}_2)}$, ${t_l}^2 \in \mathcal{N}_g (q_{os})$ for 
every simple closed curve $l$ on $N_g$. 
For $a \in H_1(N_g; \mathbb{Z}_2)$ with $q_{os}(a) = 1$, 
$T_a$ preserves $q_{os}$, hence $t_{d_i} \in \mathcal{N}_g (q_{os})$. 
For $a, b \in H_1(N_g; \mathbb{Z}_2)$ with $q_{os}(a) = q_{os}(b) = q_{os}(a+b) =0$, 
$T_a T_b T_{a+b}$ preserves $q_{os}$, hence 
$t_{a_i} t_{a_{i+2}} t_{c_i}$ are elements of $\mathcal{N}_g (q_{os})$. 
Since $(Y_{i,j})_* = id_{H_1(N_g; \mathbb{Z}_2)}$, 
$\mathcal{YS}_g \subset \mathcal{N}_g(q_{os})$.  
Therefore, in order to prove Theorem \ref{thm:generator-pin}, 
we should see that every element of $\mathcal{N}_g (q_{os})$ 
is a product of these elements. 
\subsection{Short-leg Y-homeomorphisms}
For a Y-homeomorphism $Y_{m,a}$, we call $m$ {\em the leg\/} of $Y_{m,a}$ 
and $a$ {\em the arm\/} of $Y_{m,a}$. 
A Y-homeomorphism is called a {\em short-leg\/} Y-homeomorphism, 
if its leg is one of $x_1$, $\ldots$, $x_g$ illustrated in 
Figure \ref{fig:o-standard-basis}. 
\begin{lem}\label{lem:short-leg-gen}
%
Every short-leg Y-homeomorphism is an element of $\mathcal{YS}_g$. 
\end{lem}
\begin{proof}
\begin{figure}[hbtp]
\includegraphics[height=2cm]{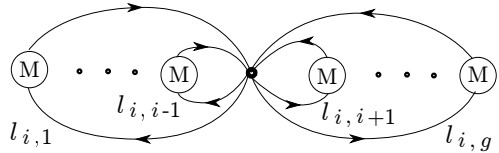}
\caption{
%
The generators for $\pi_1(N_g, p_0)$. 
}
\label{fig:fund-gen}
\end{figure}
We review the {\it crosscap pushing map\/} defined in \cite{Szepietowski2}. 
Fix $p_0 \in N_{g-1}$ and define $\mathcal{M}(N_{g-1}, p_0)$ be the group 
of isotopy classes of diffeomorphisms over $N_{g-1}$ preserving $p_0$. 
Let $U$ be a $2$-disk embedded in $N_{g-1}$ such that the center of $U$ is $p_0$. 
We define a homomorphism $j$ from $\pi_1 (N_{g-1}, p_0)$ to 
$\mathcal{M}(N_{g-1}, p_0)$ such that, for a loop $\gamma$ in $N_g$ 
based at $x_0$ and an element $[\gamma] \in \pi_1 (N_{g-1}, x_0)$, 
$j([\gamma])$ is a diffeomorphism over $N_g$ obtained as the effect of 
pushing $p_0$ once along $\gamma$. 
This homomorphism $j$ is in a non-orientable analogy of 
the Birman exact sequence \cite{Birman}. 
We define a homomorphism $\varphi$ from $\mathcal{M}(N_{g-1},p_0)$ 
to $\mathcal{M}(N_g)$ as follows. 
We represent $h \in \mathcal{M}(N_{g-1},p_0)$ by a diffeomorphism $h$ 
over $N_g$ such that $h(U) = U$ and $h|_U = id_U$. 
We construct $N_g$ from $N_{g-1} - int\, U$ by attaching a M\"obius band along 
$\partial U$. Here we assume that this M\"obius band is 
the $i$-th band on $N_g$. 
We extend $h |_{N- int\, U}$ to a diffeomorphism $\varphi(h)$ over $N_g$ 
constructed as above by the identity on the M\"obius band. 
The homomorphism $\psi = \varphi \circ j$ 
is called a crosscap pushing map. 

Every short-leg Y-homeomorphism $Y_{x_i, a}$ 
is in $\psi(\pi_1 (N_{g-1}, p_0))$, $\pi_1(N_{g-1}, p_0)$ 
is generated by the loops $l_{i,j}$'s indicated in Figure \ref{fig:fund-gen}, 
and $\psi(l_{i,j}) = Y_{i,j}$, hence $Y_{x_i, a}$ is a product of $Y_{i,j}$'s. 
\end{proof}
Let $G_g$ be the subgroup of $\mathcal{M}(N_g)$ generated by 
$\mathcal{YS}_g$, 
$t_{a_1}^2$, $\ldots$, $t_{a_{g-1}}^2$, 
$t_{c_1}^2$, $\ldots$, $t_{c_{g-3}}^2$, 
$t_{d_1}$, $\ldots$, $t_{d_{g-2}}$, 
$t_{a_1} t_{a_3} t_{c_1}$, $\ldots$, 
$t_{a_{g-3}} t_{a_{g-1}} t_{c_{g-3}}$.  
We have already shown that $G_g \subset \mathcal{N}_g(q_{os})$, 
therefore, what we should show is $\mathcal{N}_g(q_{os}) \subset G_g$. 
Two Y-homeomorphisms $Y_1$ and $Y_2$ are {\em $G_g$-equivalent\/} 
if there is an element $\phi$ of $G_g$ such that 
$\phi Y_1 \phi^{-1} = Y_2$. 
We remark that if $Y_1 = Y_{m,a}$ and $Y_2 = \phi Y_1 \phi^{-1}$ 
then $Y_2 = Y_{\phi(m), \phi(a)}$. 
We will show: 
\begin{lem}\label{lem:all-Y-hom-short}
%
Every Y-homeomorphism is a product of Y-homeomorphisms which are 
$G_g$-equivalent to short-leg Y-homeomorphisms. 
\end{lem}
By Lemmas \ref{lem:short-leg-gen} and \ref{lem:all-Y-hom-short}, 
we see that every Y-homeomorphism is an element of $G_g$. 
Therefore, by Theorem \ref{thm:level2}, we conclude: 
\begin{cor}\label{cor:level-2-in-Gg}
%
$\Gamma_2(N_g) \subset G_g$. \qed
\end{cor}
\begin{rem}
While the author was writing this paper, 
Szepietowski informed the author that he found 
a finite system of generators for $\Gamma_2(N_g)$. 
In the next subsection, 
we introduce his system of generators and prove 
Lemma \ref{lem:all-Y-hom-short} by using his result. 
In this subsection, we show Lemma \ref{lem:all-Y-hom-short} 
by our original proof. 
\end{rem}
\begin{figure}[hbtp]
\includegraphics[height=2.5cm]{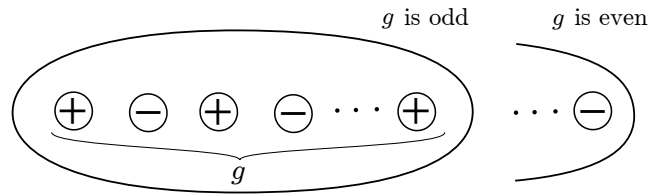}
\caption{
%
Diagram indicating $o$-standard $N_g$ in $S^4$. 
}
\label{fig:o-standard-pm}
\end{figure}
As shown in Figure \ref{fig:o-standard-pm}, 
we use the symbol $\oplus$ (resp. $\ominus$) 
to indicate the place where the M\"{o}bius band are attached such that 
$q_{os} (x_i)= +1$ (resp. $q_{os} (x_i)= -1$)  
for the circle $x_i$ indicated in Figure \ref{fig:o-standard-basis}. 
We denote an element $x = \sum_{i=1}^g \epsilon_i x_i$ 
$\in H_1(N_g ; \mathbb{Z}_2)$, 
where $\epsilon_i = 0$ or $1$, by a sequence 
of symbols $+, -, \oplus, \ominus$ of length $g$ with $[,]$ 
which are settled by the rule: 
if $\epsilon_{2i-1} = 0$ then the $2i-1$-st symbol is $+$, 
if $\epsilon_{2i-1} = 1$ then the $2i-1$-st symbol is $\oplus$, 
if $\epsilon_{2i} = 0$ then the $2i$-th symbol is $-$, 
and if $\epsilon_{2i} = 1$ then the $2i$-th symbol is $\ominus$. 
For example, when $g=7$, we denote an element $x_2 + x_3 + x_6 + x_7$ 
by $[+ \ominus \oplus - + \ominus \oplus]$. 
This sequence is called the {\em r-sequence\/} associated to $x$. 
For the r-sequence associated to $x$, we settle a simple closed curve on $N_g$ 
by the following rule. 
\begin{figure}[hbtp]
\includegraphics[height=2.5cm]{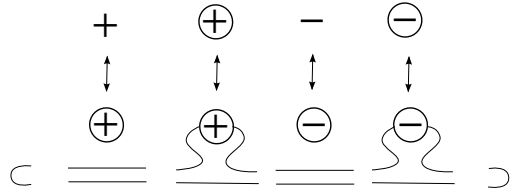}
\caption{
%
Parts of r-circles. 
}
\label{fig:parts}
\end{figure}
For the symbols in this sequence, we put arcs on $N_g$ 
indicated in the bottom of Figure \ref{fig:parts}, glue them along the boundaries, 
and cap by the arc indicated on the left of Figure \ref{fig:parts} from the left 
and by the arc indicated on the right of Figure \ref{fig:parts} from the right. 
We call this circle the {\em r-circle\/} associated to $x$ 
and denote by $R(x)$. 
For an element $x = \sum_{i=1}^g \epsilon_i x_i$ 
$\in H_1(N_g ; \mathbb{Z}_2)$, 
where $\epsilon_i = 0$ or $1$, 
we define $\text{supp}(x) = \{ x_i \ | \ \epsilon_i = 1\}$.  

Two simple closed curves $c_1$ and $c_2$ on $N_g$ are {\em $G_g$-equivalent\/} 
($c_1 \underset{G_g}{\sim} c_2$) 
if there is an element $\phi$ of $G_g$ such that $\phi(c_1) = c_2$. 
\begin{lem}\label{lem:G-g-eq-r-circle} 
%
If $g=1$, then every r-circle is $G_g$-equivalent to $R([+])$ or $R([\oplus])$. 
If $g=2$, then every r-circle is $G_g$-equivalent to $R([+ -])$, 
$R([\oplus -])$, $R([+ \ominus])$ or $R([\oplus \ominus])$. 
If $g \geq 3$ is odd, then every r-circle is $G_g$-equivalent to 
$R([+ - + - \cdots +])$, $R([\oplus - + - \cdots +])$, 
$R([+ \ominus + - \cdots +])$, $R([\oplus \ominus + - \cdots +])$, 
$R([\oplus - \oplus - \cdots +])$ or 
$R([\oplus \ominus \oplus \ominus \cdots \oplus])$. 
If $g \geq 4$ is even, then every r-circle is $G_g$-equivalent to 
$R([+ - + - \cdots -])$, $R([\oplus - + - \cdots -])$, 
$R([+ \ominus + - \cdots -])$, $R([\oplus \ominus + - \cdots -])$, 
$R([\oplus - \oplus - \cdots -])$ or 
$R([\oplus \ominus \oplus \ominus \cdots \ominus])$. 
\end{lem}
\begin{proof}
If $g=1$ or $2$, then the conclusion is trivial. \\
If $g \geq 3$, then 
%
\begin{equation}
\begin{aligned}\label{eqn:G-g-equiv-1}
R([\cdots - + \ominus \cdots]) &\underset{G_g}{\sim} 
R([\cdots \ominus + - \cdots]), \\
R([\cdots + - \oplus \cdots]) &\underset{G_g}{\sim} 
R([\cdots \oplus - + \cdots]), \\
R([\cdots - \oplus \ominus \cdots]) &\underset{G_g}{\sim} 
R([\cdots \ominus \oplus - \cdots]), \\ 
R([\cdots + \ominus \oplus \cdots]) &\underset{G_g}{\sim} 
R([\cdots \oplus \ominus + \cdots]),
\end{aligned}
\end{equation}
where the left most symbols are the $i$-th symbol, 
since 
$Y_{i+2,i} Y_{i+1,i} t_{d_i} R([\cdots - + \ominus \cdots]) 
= R([\cdots \ominus + - \cdots])$, 
$Y_{i+2,i} Y_{i+1,i} t_{d_i} R([\cdots + - \oplus \cdots])  
= R([\cdots \oplus - + \cdots])$, 
$Y_{i+2,i+1} Y_{i+1,i+2} t_{d_i} R([\cdots - \oplus \ominus \cdots]) 
= R([\cdots \ominus \oplus - \cdots])$ and 
$Y_{i+2,i+1} Y_{i+1,i+2} t_{d_i} R([\cdots + \ominus \oplus \cdots]) 
= R([\cdots \oplus \ominus + \cdots])$. 
Therefore, when $g=3$, for two cases $R([+ - \oplus])$ and 
$R([+ \ominus \oplus])$ which are not listed in the statement, 
we see $R([+ - \oplus]) \underset{G_g}{\sim} R([\oplus - +])$ 
and $R([+ \ominus \oplus]) \underset{G_g}{\sim} R([\oplus \ominus +])$. 

If $g \geq 4$, then  
%
\begin{equation}
\begin{aligned}\label{eqn:G-g-equiv-2}
R([\cdots - \oplus \ominus \oplus \cdots]) &\underset{G_g}{\sim} 
R([\cdots - \oplus - + \cdots]), \\
R([\cdots + \ominus \oplus \ominus \cdots]) &\underset{G_g}{\sim} 
R([\cdots + \ominus + - \cdots]), \\
R([\cdots \ominus \oplus \ominus + \cdots]) &\underset{G_g}{\sim} 
R([\cdots - + \ominus + \cdots]), \\ 
R([\cdots \oplus \ominus \oplus - \cdots]) &\underset{G_g}{\sim} 
R([\cdots + - \oplus - \cdots]), \\
R([\cdots - \oplus - \oplus \cdots]) &\underset{G_g}{\sim} 
R([\cdots \ominus + \ominus + \cdots]), \\ 
R([\cdots + \ominus + \ominus \cdots]) &\underset{G_g}{\sim} 
R([\cdots \oplus - \oplus - \cdots]), \\
\end{aligned}
\end{equation}
where the left most symbols are the $i$-th symbol, 
since \\
$Y_{i+3,i+1} Y_{i+2,i+1} t_{a_i}^{-2} (t_{a_i} t_{a_{i+2}} t_{c_i}) 
R([\cdots - \oplus \ominus \oplus \cdots]) = 
R([\cdots - \oplus - + \cdots])$, \\
$Y_{i+3,i+1} Y_{i+2,i+1} t_{a_i}^{-2} (t_{a_i} t_{a_{i+2}} t_{c_i}) 
R([\cdots + \ominus \oplus \ominus \cdots]) = 
R([\cdots + \ominus + - \cdots])$, \\
$Y_{i,i+2} Y_{i+1,i+2} t_{a_{i+2}}^2 (t_{a_i} t_{a_{i+2}} t_{c_i})^{-1} 
R([\cdots \ominus \oplus \ominus + \cdots]) = 
R([\cdots - + \ominus + \cdots])$, \\
$Y_{i,i+2} Y_{i+1,i+2} t_{a_{i+2}}^2 (t_{a_i} t_{a_{i+2}} t_{c_i})^{-1} 
R([\cdots \oplus \ominus \oplus - \cdots]) =  
R([\cdots + - \oplus - \cdots])$, \\
$ Y_{i+1,i} (t_{a_i} t_{a_{i+2}} t_{c_i}) Y_{i+2,i+3}^{-1} 
R([\cdots - \oplus - \oplus \cdots]) =  
R([\cdots \ominus + \ominus + \cdots])$, \\
$ Y_{i+1,i} (t_{a_i} t_{a_{i+2}} t_{c_i}) Y_{i+2,i+3}^{-1} 
R([\cdots + \ominus + \ominus \cdots]) = 
R([\cdots \oplus - \oplus - \cdots])$. 

When $g \geq 4$, we get our conclusion by the induction on $g$ and $G_g$-equivalences 
(\ref{eqn:G-g-equiv-1}) and (\ref{eqn:G-g-equiv-2}).
\end{proof}
If the complement of an M-circle $m$ is orientable, then any circle intersecting $m$ 
transversely in one point is an M-circle. 
Therefore the leg of evey Y-homeomorphism is an M-circle whose complement is 
non-orientable. 
Every element of $G_g$ preserves $q_{os}$, 
the r-circles $R([+ - \cdots \pm])$, $R([\oplus - \oplus - \cdots \pm])$ 
and $R([\oplus \ominus + \cdots \pm])$ are A-circles, 
and the complements of $R([\oplus \ominus \cdots \oplus])$ and 
$R([\oplus \ominus \cdots \ominus])$ are orientable, hence we see: 
\begin{cor}\label{cor:G-g-eq-r-circle}
%
If an r-circle $R(x)$ is a leg of a Y-homeomorphism, 
then $R(x)$ is $G_g$-equivalent to 
$R([\oplus - + \cdots])$ or $R([+ \ominus + \cdots])$. 
\qed
\end{cor}
By investigating the action of Chillingworth's generators for $\mathcal{M}(N_g)$ on 
legs of Y-homeomoprhisms, we see: 
\begin{lem}\label{lem:product-Y-homeo} 
%
Every Y-homeomorphism is a product of Y-homeomorphisms 
whose legs are r-circles. 
\end{lem} 
\begin{proof}
Since $\{ x_1, \ldots, x_g \}$ are r-circles, $Y_{i,j}$ is 
a Y-homeomorphism whose leg is an r-circle. 
For every Y-homeomorphism $Y_{m,a}$, there is an r-circle $s$ and an element 
$\phi \in \mathcal{M}(N_g)$ such that $\phi(s) = m$, that is, 
$Y_{m,a} = \phi \ Y_{s, \phi^{-1}(a)} \ \phi^{-1}$. 
Therefore, by Lemma \ref{lem:generator}, 
it suffices to show that  
there are $\phi_i, \phi'_i, \psi_i, \psi'_i$ $\in \mathcal{YS}_g$ 
and r-circles $s_i,s'_i, t_i, t'_i$ such that 
$t_{a_i} (s) = \phi_i (s_i)$, $t_{a_i}^{-1} (s) = \phi'_i (s'_i)$, 
$t_{c_i} (s) = \psi_i (t_i)$ and $t_{c_i}^{-1} (s) = \psi'_i (t'_i)$ 
for every r-circle $s$. 
As observed in the proof of \cite[Lemma 3.1]{Szepietowski}, 
$Y_{x_i,a_i}$ preserves $a_i$ and exchanges the sides of $a_i$, 
hence $t_{a_i}= Y_{x_i,a_i} t_{a_i}^{-1} Y_{x_i,a_i}^{-1}$, 
therefore 
$t_{a_i}^2 = t_{a_i} Y_{x_i,a_i} t_{a_i}^{-1} Y_{x_i,a_i}^{-1}$
$= Y_{t_{a_i}(x_i),a_i} Y_{x_i,a_i}^{-1}$ 
$= Y_{i+1,i} Y_{i,i+1}^{-1}$. 
By the same way as above, we see that 
$t_{c_i} = Y_{x_i,c_i} t_{c_i}^{-1} Y_{x_i, c_i}^{-1}$, 
therefore 
$t_{c_i}^2 = t_{c_i} Y_{x_i,c_i} t_{c_i}^{-1} Y_{x_i,c_i}^{-1}$
$= Y_{t_{c_i}(x_i),c_i} Y_{x_i,c_i}^{-1}$. 
Since $Y_{g,i+3} \cdots Y_{i+4,i+3} Y_{i,i+1} \cdots Y_{1,i+1} t_{c_i}(x_i)$ 
is isotopic to $R(x_{i+1} + x_{i+2} + x_{i+3})$, 
$t_{c_i}^2$ is a products of Y-homeomorphisms whose legs are r-circles. 
From the above observation, 
it suffices to show that one of $t_{a_i} (s) = \phi_i (s_i)$, 
$t_{a_i}^{-1} (s) = \phi'_i (s'_i)$, and 
one of $t_{c_i} (s) = \psi_i (t_i)$, $t_{c_i}^{-1} (s) = \psi'_i (t'_i)$. 

Since $a_i$ does not intersects $R(x)$ such that 
$\text{supp}(x) \cap \{x_i, x_{i+1} \} = \emptyset$, 
we only consider the action of $t_{a_i}$ on $R(x)$ 
such that $\text{supp}(x) \cap \{x_i, x_{i+1} \} \not= \emptyset$. 
When we consider the action of $t_{a_i}$ and Y-homeomorphisms, we do not need to 
take care of the sign on the M\"{o}bius bands. 
Hence, in symbols of r-sequences, we change $+$ and $-$ into $\times$, 
and $\oplus$ and $\ominus$ into $\otimes$. 
There are 3 cases to consider: $R([\cdots \otimes \times \cdots])$, 
$R([\cdots \times \otimes \cdots])$, and 
$R([\cdots \otimes \otimes \cdots])$, 
where the $i$-th and $i+1$-st symbols are indicated. 
The 3-rd r-circle does not intersect $a_i$, hence we ignore this. 
By drawing figures of r-circles, we see: 
$Y_{i,i+1}(t_{a_i}^{-1}(R([\cdots \otimes \times \cdots]))) 
= R([\cdots \times \otimes \cdots])$ and 
$Y_{i+1,i}(t_{a_i}(R([\cdots \times \otimes \cdots]))) 
=R([\cdots \otimes \times \cdots])$. 

By the same reasons as in the previous paragraph, 
it suffice to consider the action of $t_{c_i}$ on $R(x)$ 
such that $\text{supp}(x) \cap 
\{x_i, x_{i+1}, x_{i+2}, x_{i+3} \} \not= \emptyset$, and, 
in symbols of r-sequences, we change $+$ and $-$ into $\times$, 
and $\oplus$ and $\ominus$ into $\otimes$. 
There are 15 cases to consider: 
(1) $R([\cdots \otimes \times \times \times \cdots])$, 
(2) $R([\cdots \times \otimes \times \times \cdots])$, 
(3) $R([\cdots \otimes \otimes \times \times \cdots])$, 
(4) $R([\cdots \times \times \otimes \times \cdots])$, 
(5) $R([\cdots \otimes \times \otimes \times \cdots])$, 
(6) $R([\cdots \times \otimes \otimes \times \cdots])$, 
(7) $R([\cdots \otimes \otimes \otimes \times \cdots])$, 
(8) $R([\cdots \times \times \times \otimes \cdots])$, 
(9) $R([\cdots \otimes \times \times \otimes \cdots])$, 
(10) $R([\cdots \times \otimes \times \otimes \cdots])$, 
(11) $R([\cdots \otimes \otimes \times \otimes \cdots])$, 
(12) $R([\cdots \times \times \otimes \otimes \cdots])$, 
(13) $R([\cdots \otimes \times \otimes \otimes \cdots])$, 
(14) $R([\cdots \times \otimes \otimes \otimes \cdots])$, 
(15) $R([\cdots \otimes \otimes \otimes \otimes \cdots])$, 
where the $i$-th, $i+1$-st, $i+2$-nd and $i+3$-rd symbols are indicated.
Since (3) (6) (12) and (15) do not intersect $c_i$, 
$t_{c_i}$ does not change these r-circles. 
By drawing figures of r-circles, 
we see: \\
(1) $Y_{i,i+1} Y_{i+2,i+3} Y_{i+1,i+3} Y_{i+1,i+2}^{-1}
(t_{c_i}^{-1}(R([\cdots \otimes \times \times \times \cdots])))
=R([\cdots \times \otimes \otimes \otimes \cdots])$, \\
(2) $Y_{i+1,i} Y_{i+2,i+3} 
(t_{c_i}(R([\cdots \times \otimes \times \times \cdots])))
=R([\cdots \otimes \times \otimes \otimes \cdots])$, \\
(4) $ Y_{i+2,i+3} Y_{i+1,i}
(t_{c_i}^{-1}(R([\cdots \times \times \otimes \times \cdots])))
=R([\cdots \otimes \otimes \times \otimes \cdots])$, \\
(5) $ Y_{i+1,i+2} Y_{i,i+2} Y_{i+3,i+2}^{-1} Y_{i+2,i+3}^{-1} Y_{i,i+3}^{-1} 
(t_{c_i}^{-1}(R([\cdots \otimes \times \otimes \times \cdots])))
=R([\cdots \otimes \times \otimes \times \cdots])$, \\
(7) $Y_{i,i+3} Y_{i+1,i+3} Y_{i+2,i+3} 
(t_{c_i}^{-1}(R([\cdots \otimes \otimes \otimes \times \cdots])))
=R([\cdots \times \times \times \otimes \cdots])$,\\
(8) $ Y_{i+3,i+2} Y_{i+1,i} Y_{i+2,i} Y_{i+2,i+1}^{-1} 
(t_{c_i}(R([\cdots \times \times \times \otimes \cdots])))
= R([\cdots \otimes \otimes \otimes \times \cdots])$, \\
(9) $Y_{i+2,i+3} Y_{i+1,i} Y_{i+3,i} Y_{i+3,i+1}^{-1} Y_{i+3,i+2} 
Y_{i+2,i+3} Y_{i+1,i+3} Y_{i+1,i+2}^{-1} Y_{i,i+3} Y_{i,i+2}^{-1} Y_{i,i+1}
(t_{c_i}^{-1}(R([\cdots \otimes \times \times \otimes \cdots])))
=R([\cdots \otimes \times \times \otimes \cdots])$, \\
(10) $Y_{i+2,i+1} Y_{i+3,i+1} Y_{i,i+3}^{-1} Y_{i+1,i}^{-1} Y_{i+3,i}^{-1}
(t_{c_i}(R([\cdots \times \otimes \times \otimes \cdots])))
=R([\cdots \times \otimes \times \otimes \cdots])$,\\
(11) $ Y_{i,i+2} Y_{i+1,i+2} Y_{i+3,i+2} Y_{i+2,i} Y_{i+2,i+1}^{-1} 
(t_{c_i}(R([\cdots \otimes \otimes \times \otimes \cdots])))
=R([\cdots \times \times \otimes \times \cdots])$, \\
(13) $Y_{i+3,i+1} Y_{i+2,i+1} Y_{i,i+1} Y_{i+1,i+3} Y_{i+1,i+2}^{-1} 
(t_{c_i}^{-1}(R([\cdots \otimes \times \otimes \otimes \cdots])))
=R([\cdots \times \otimes \times \times \cdots])$, \\
(14) $Y_{i+3,i} Y_{i+2,i} Y_{i+1,i} 
(t_{c_i}(R([\cdots \times \otimes \otimes \otimes \cdots])))
= R([\cdots \otimes \times \times \times \cdots])$. 
\end{proof}

{\it Proof of Lemma \ref{lem:all-Y-hom-short}.}
Let $Y_{m,a}$ be a Y-homeomorphism whose leg is an r-circle. 
By Corollary \ref{cor:G-g-eq-r-circle}, there is an element 
$\phi \in G_g$ such that $\phi(m) = R([\oplus - + \cdots])$ 
or $R([+ \ominus + \cdots])$. 
Therefore $Y_{m,a}$ is $G_g$-equivalent to a short-leg Y-homeomorphism. 
By Lemma \ref{lem:product-Y-homeo}, we get our conclusion. 
\qed
\subsection{Szepietowski's generators for $\Gamma_2(N_g)$}
\begin{figure}[hbtp]
\includegraphics[height=2.7cm]{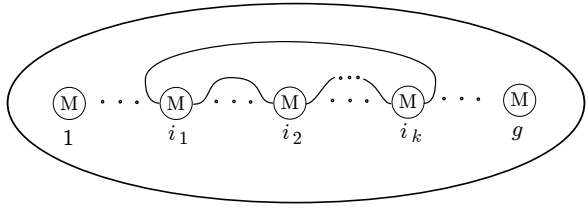}
\caption{
%
The curve $\alpha_I$ for $I = \{ i_1, i_2, \ldots, i_k \}$. 
}
\label{fig:alpha}
\end{figure}
We review the finite system of generators for $\Gamma_2(N_g)$ 
introduced in \cite{Szepietowski2}. 
For each non empty subset $I =\{i_1, i_2, \ldots, i_k \}$ of 
$\{1, \ldots, g\}$, 
let $\alpha_I$ be the simple closed curve shown in Figure \ref{fig:alpha}. 
If $I = \{ i \}$, we write $\alpha_i$ instead of $\alpha_{\{i \}}$. 
Szepietowski proved: 
\begin{thm}\cite[Theorem 3.2]{Szepietowski2}
\label{thm:gamma2-gen}
%
For $g \geq 4$, $\Gamma_2(N_g)$ is generated by the following elements. \par
1) $Y_{\alpha_i,\alpha_{\{i,j\}}}$ for $i \not= j$, \par
2) $Y_{\alpha_{\{i,j,k\}}, \alpha_{\{i,j,k,l\}}}$ for $i < j < k< l$. \\
The group $\Gamma_2(N_3)$ is generated by the elements 1). 
\end{thm}
We show: 
\begin{lem}\label{lem:gamma2-short}
%
For arbitrary $i < j < k< l$, 
$Y_{\alpha_{\{i,j,k\}}, \alpha_{\{i,j,k,l\}}}$ is 
$G_g$-equivalent to a short-leg Y-homeomorphism. 
\end{lem}
\begin{proof}
It suffices to show that, for every $i < j < k$, 
$\alpha_{\{i,j,k\}}$ is $G_g$-equivalent to $\alpha_1$ or $\alpha_2$. 
By drawing figures, we see: 
when $i > 2$, 
$Y_{i,i-2}^{-1} Y_{i-1,i-2}^{-1} t_{d_{i-2}} \alpha_{i,j,k} 
= \alpha_{i-2,j,k}$; 
when $i < j-2$, 
$Y_{j,j-2}^{-1} Y_{j-1,j-2}^{-1} t_{d_{j-2}} \alpha_{i,j,k} 
= \alpha_{i,j-2,k}$; 
when $j < k-2$, 
$Y_{k,k-2}^{-1} Y_{k-1,k-2}^{-1} t_{d_{k-2}} \alpha_{i,j,k} 
= \alpha_{i,j,k-2}$. 
By applying the above action of $G_g$ on $\alpha_{i,j,k}$, 
we see that $\alpha_{i,j,k}$ is $G_g$-equivalent to 
$\alpha_{1,3,4}$, $\alpha_{1,2,3}$, $\alpha_{2,3,5}$, $\alpha_{2,4,6}$, 
$\alpha_{1,3,5}$, $\alpha_{1,2,4}$, $\alpha_{2,3,4}$ or $\alpha_{2,4,5}$. 
By drawing figures of the action of $G_g$ on the above 8 circles, 
we can check that former 4 circles are $G_g$-equivalent to 
$\alpha_1$ and last 4 circles are $G_g$-equivalent to $\alpha_2$. 
\end{proof}
Every Y-homeomorphism $Y$ is an element of $\Gamma_2(N_g)$, 
and by Theorem \ref{thm:gamma2-gen} and Lemma \ref{lem:gamma2-short}, 
we can express $Y$ as a product of Y-homeomorphisms which are $G_g$-equivalent to 
short-leg Y-homeomorphisms. 
Hence, Lemma \ref{lem:all-Y-hom-short} follows. 

\subsection{Generators for $\mathcal{O}_g (q_{os})$}
If we find a system of generators $\{ S_1, \ldots, S_k \}$ for 
$\mathcal{O}_g (q_{os})$ and 
elements $\sigma_1, \ldots, \sigma_k$ of $\mathcal{M}(N_g)$ such that 
$(\sigma_i)_* = S_i$ in $Aut(H_1(N_g ; \mathbb{Z}_2))$, then, 
by the short exact sequence (\ref{eq:short-exact}) and Corollary \ref{cor:level-2-in-Gg}, 
we see that $\mathcal{N}_g(q_{os})$ is generated by 
$G_g \cup \{\sigma_1, \ldots, \sigma_k\}$. 
Nowik showed in \cite[Theorem 3.2.]{Nowik}: 
\begin{thm}\label{thm:generator-Og}
%
$\mathcal{O}_g (q_{os})$ is generated by the set of elements of the following 
two forms: \\
1. $T_a$ for $a \in H_1(N_g ; \mathbb{Z}_2)$ with $q_{os}(a)=2$, \\
2. $T_a \ T_b \ T_{a+b}$ for $a, b \in H_1(N_g ; \mathbb{Z}_2)$ with 
$q_{os}(a) = q_{os}(b) = q_{os}(a+b) = 0$. \\
\end{thm}

Let $\{ x_1, \ldots, x_g \}$ be the basis of $H_1(N_g ; \mathbb{Z}_2)$ 
which is introduced in Figure \ref{fig:o-standard-basis}. 
We obtain a finite system of generators for $\mathcal{O}_g(q_{os})$ 
explicitly. 
\begin{lem}\label{lem:gen-Og-os-red}
%
$\mathcal{O}_g (q_{os})$ is generated by 
\begin{align}
&T_{x_i+x_{i+2}} \ (i=1, \ldots, g-2), \label{elm:2}\\
&T_{x_i+x_{i+1}} T_{x_{i+2}+x_{i+3}} T_{x_i+x_{i+1} + x_{i+2}+x_{i+3}} 
\ (i=1, \ldots, g-3). \label{elm:3}
\end{align}
\end{lem}
\begin{proof}
We write any element $v$ of $H_1(N_g;\mathbb{Z}_2)$ 
as $v = x_{i_1} + \ldots + x_{i_m}$ such that $i_1 < \ldots <i_m$ 
and call $m$ {\em the length\/} of $v$, 
or as $v = (x_{2j_1 + 1} + \cdots + x_{2j_k + 1}) \oplus 
(x_{2j_{k+1}} + \cdots + x_{2j_m})$ such that 
$j_1 \leq \cdots \leq j_k$, $j_{k+1} \leq \cdots \leq j_m$ 
and call $(x_{2j_1 + 1} + \cdots + x_{2j_k + 1})$ 
{\em the odd part\/} of $v$, and 
$(x_{2j_{k+1}} + \cdots + x_{2j_m})$ 
{\em the even part\/} of $v$. 
Two elements $v$, $w$ of $H_1(N_g;\mathbb{Z}_2)$ is 
{\em (\ref{elm:2})-equivalent\/} $v \overset{\ref{elm:2}}{\sim}w$ 
if there is a product $T$ of (\ref{elm:2}) such that $T(v) = w$, 
and define {\em (\ref{elm:3})-equivalence\/} 
$v \overset{\ref{elm:3}}{\sim} w$ and 
{\em (\ref{elm:2}) and (\ref{elm:3})-equivalence\/} 
$v \overset{\ref{elm:2},\ref{elm:3}}{\sim} w$ in the same way. 
We remark that if $v \overset{\ref{elm:2},\ref{elm:3}}{\sim} w$ 
then there is a product $T$ of (\ref{elm:2}) and (\ref{elm:3}) 
such that $T_w = T T_v T$. 

Any element (\ref{elm:2}) acts only on the odd part of $v$ or only 
on the even part of $v$. 
For example, when $i < j < k$, 
$$
\begin{aligned}
T_{x_{2j-1} + x_{2j+1}} 
((\cdots + x_{2i-1} + x_{2j+1} + &x_{2k+1} + \cdots) \oplus (\cdots)) = \\
&((\cdots + x_{2i-1} + x_{2j-1} + x_{2k+1} + \cdots) \oplus (\cdots)). 
\end{aligned}
$$
Therefore, if we define $l_o (v) =$ the length of the odd part of $v$, 
and $l_e (v)=$ the length of the even part of $v$, 
then $v \overset{\ref{elm:2}}{\sim}
(x_1 + x_3 + \cdots + x_{2l_o(v)-1}) \oplus (x_2 + x_4 + \cdots + x_{2l_e(v)})$. 
Hence, $v \overset{\ref{elm:2}}{\sim} w$ 
if and only if $l_o(v) = l_o(w)$ and $l_e(v) = l_e(w)$. 

When $p<i, i+3<s$, $\{ i, i+1 \} = \{q, q'\}$ 
and $\{ i+2, i+3 \} = \{r, r'\}$, 
the element (\ref{elm:3}) acts as follows. 
$$
\begin{aligned}
&T_{x_{i} + x_{i+1}} T_{x_{i+2}+x_{i+3}} 
T_{x_{i} + x_{i+1} + x_{i+2} + x_{i+3}} 
(\cdots + x_p + x_q + x_r + x_s + \cdots)  \\
&= T_{x_{i} + x_{i+1}} T_{x_{i+2}+x_{i+3}}  
(\cdots + x_p + x_q + x_r + x_s + \cdots)  \\
&= \cdots + x_p + x_{q'} + x_{r'} + x_s + \cdots,  \\
&T_{x_{i} + x_{i+1}} T_{x_{i+2}+x_{i+3}} 
T_{x_{i} + x_{i+1} + x_{i+2} + x_{i+3}} 
(\cdots + x_p + x_q + x_s + \cdots)  \\
&= T_{x_{i} + x_{i+1}} T_{x_{i+2}+x_{i+3}}  
(\cdots + x_p + x_{q'} + x_r + x_{r'} + x_s + \cdots) \\
&= \cdots + x_p + x_q + x_r + x_{r'} + x_s + \cdots, \\
&T_{x_{i} + x_{i+1}} T_{x_{i+2}+x_{i+3}} 
T_{x_{i} + x_{i+1} + x_{i+2} + x_{i+3}} 
(\cdots + x_p + x_r + x_s + \cdots)  \\
&= T_{x_{i} + x_{i+1}} T_{x_{i+2}+x_{i+3}}  
(\cdots + x_p + x_q + x_{q'} + x_{r'} + x_s + \cdots) \\
&= \cdots + x_p + x_q + x_{q'}+ x_r + x_s + \cdots.  
\end{aligned} 
$$

We will show that every element of the first form in Theorem \ref{thm:generator-Og} 
is a product of (\ref{elm:2}) and (\ref{elm:3}). 
Let $a$ be an element of $H_1(N_g ; \mathbb{Z}_2)$ such that $q_{os}(a)=2$. 
Then, $2 \equiv l_o(a) - l_e(a) \mod 4$. 
Therefore, there are two cases $l_o(a) = l_e(a) + 4t + 2$ or 
$l_e(a) = l_o(a) + 4t +2$ ($t \in \mathbb{Z}_{\geq 0}$). 
For the first case, 
$$
\begin{aligned}
a &\overset{\ref{elm:2}}{\sim} 
(x_1 + x_3) + (x_4 + x_5) + \cdots + (x_{2i} + x_{2i+1}) + 
(x_l + x_{l+2} + x_{l+4} + x_{l+6}) + \cdots \\
& \cdots + (x_m + x_{m+2} + x_{m+4} + x_{m+6}). 
\end{aligned}
$$
For the second case, 
$$
\begin{aligned}
a &\overset{\ref{elm:2}}{\sim} 
(x_2 + x_4) + (x_5 + x_6) + \cdots + (x_{2i+1} + x_{2i+2}) + 
(x_l + x_{l+2} + x_{l+4} + x_{l+6}) + \cdots \\
& \cdots + (x_m + x_{m+2} + x_{m+4} + x_{m+6}). 
\end{aligned}
$$  
We see 
$$
\begin{aligned}
&(x_2 + x_4) + (x_5 + x_6) + \cdots + (x_{2i+1} + x_{2i+2}) + 
(x_l + x_{l+2} + x_{l+4} + x_{l+6}) + \cdots \\
& \cdots + (x_m + x_{m+2} + x_{m+4} + x_{m+6}) \overset{\ref{elm:3}}{\sim} \\
&(x_1 + x_3) + (x_5 + x_6) + \cdots + (x_{2i+1} + x_{2i+2}) + 
(x_l + x_{l+2} + x_{l+4} + x_{l+6}) + \cdots \\
& \cdots + (x_m + x_{m+2} + x_{m+4} + x_{m+6}) \underset{\ref{elm:2}}{\sim} \\
&(x_1 + x_3) + (x_4 + x_5) + \cdots + (x_{2i} + x_{2i+1}) + 
(x_l + x_{l+2} + x_{l+4} + x_{l+6}) + \cdots \\
& \cdots + (x_m + x_{m+2} + x_{m+4} + x_{m+6}),  
\end{aligned}
$$
where $\overset{\ref{elm:3}}{\sim}$ is by 
$T_{x_1+x_2} T_{x_3+x_4} T_{x_1 + x_2 + x_3 + x_4}$. 
Therefore, it suffices to consider the first case. 
We see 
$
(x_l + x_{l+2} + x_{l+4} + x_{l+6}) \overset{\ref{elm:3}}{\sim} 
(x_l + x_{l+2} + x_{l+3} + x_{l+5}) \overset{\ref{elm:3}}{\sim} 
(x_l + x_{l+5}) \overset{\ref{elm:2}}{\sim} 
(x_l + x_{l+1})
$,
where the first $\overset{\ref{elm:3}}{\sim}$ 
is by $T_{x_{l+3}+x_{l+4}} T_{x_{l+5}+x_{l+6}} 
T_{x_{l+3} + x_{l+4} + x_{l+5} + x_{l+6}}$, 
and the second $\overset{\ref{elm:3}}{\sim}$ 
is by $T_{x_{l+2}+x_{l+3}} T_{x_{l+4}+x_{l+5}} 
T_{x_{l+2} + x_{l+3} + x_{l+4} + x_{l+5}}$. 
Hence, 
$ a \overset{\ref{elm:2}, \ref{elm:3}}{\sim} 
(x_1 + x_3) + (x_4 + x_5) + \cdots + (x_{2n} + x_{2n+1})$. 
Furthermore, 
$ (x_1 + x_3) + (x_4 + x_5) + (x_6 + x_7) + \cdots + (x_{2n} + x_{2n+1}) 
= (x_1 + x_3 + x_4) + x_5 + (x_6 + x_7) + \cdots + (x_{2n} + x_{2n+1}) 
\overset{\ref{elm:3}}{\sim} 
 x_1 + x_5 + (x_6 + x_7) + \cdots + (x_{2n} + x_{2n+1}) 
\overset{\ref{elm:2}}{\sim} 
 x_1 + x_3 + (x_4 + x_5) + \cdots + (x_{2(n-1)} + x_{2(n-1)+1}),  
$
where $\overset{\ref{elm:3}}{\sim}$ is by 
$T_{x_1+x_2} T_{x_3+x_4} T_{x_1+x_2+x_3+x_4}$. 
Therefore, by the induction on $n$, we see 
$ a \overset{\ref{elm:2}, \ref{elm:3}}{\sim} x_1 + x_3$. 
Hence $T_a$ is a product of (\ref{elm:2}) and (\ref{elm:3}) if $q_{os}(a)=2$.  

We will show that every element of the second form in Theorem \ref{thm:generator-Og} 
is a product of (\ref{elm:2}) and (\ref{elm:3}). 
Let $a$ and $b$ be elements of $H_1(N_g ; \mathbb{Z}_2)$ such that 
$q_{os}(a) = q_{os}(b) = q_{os}(a+b) = 0$. 
Then $0 = q_{os} (a+b) = q_{os}(a) + q_{os}(b) + (a \cdot b)_2 = (a \cdot b)_2$, 
$0 = q_{os}(2a) = q_{os} (a) + q_{os} (a) + (a \cdot a)_2 = (a \cdot a)_2$, 
by the same reason, $0 = (b \cdot b)_2$, 
hence $(a+b \cdot a)_2 = (a+b \cdot b)_2 = 0$. 
Therefore, $T_a$, $T_b$ and $T_{a+b}$ commute each other. 
For the pairs $[a_1, b_1]$ and $[a_2, b_2]$ of elements of 
$H_1(N_g ; \mathbb{Z}_2)$ which satisfies 
$q_{os} (a_i) = q_{os} (b_i) = q_{os} (a_i+b_i) = 0$ ($i=1,2$), 
we define the equivalence 
$[a_1, b_1] \overset{\ref{elm:2}}{\sim} [a_2, b_2]$ 
if there is a product $T$ of (\ref{elm:2}) such that 
$T(a_1) = a_2$ and $T(b_1) = b_2$. 
The equivalences $[a_1, b_1] \overset{\ref{elm:3}}{\sim} [a_2, b_2]$ 
and $[a_1, b_1] \overset{\ref{elm:2},\ref{elm:3}}{\sim} [a_2, b_2]$ 
are defined in the same way. 

Let $a$, $b$ be elements of $H_1(N_g ; \mathbb{Z}_2)$ 
such that $q_{os}(a) = q_{os}(b) = q_{os}(a+b)=0$. 
By the same argument applied for elements of the first from in 
Theorem \ref{thm:generator-Og}, 
we see $a \overset{\ref{elm:2},\ref{elm:3}}{\sim} 
(x_1 + x_2) + \cdots +(x_{2n-1} + x_{2n})$. 

If $2n \not= g$, then 
$(x_1 + x_2) + \cdots +(x_{2n-3} + x_{2n-2})+(x_{2n-1} + x_{2n}) 
= (x_1 + x_2) + \cdots + x_{2n-3} + (x_{2n-2}+ x_{2n-1} + x_{2n}) 
\overset{\ref{elm:3}}{\sim} 
(x_1 + x_2) + \cdots + x_{2n-3} + x_{2n} 
\overset{\ref{elm:2}}{\sim}
(x_1 + x_2) + \cdots + x_{2n-3} + x_{2n-2} 
=(x_1 + x_2) + \cdots + (x_{2(n-1)-1} + x_{2(n-1)}),   
$
where $\overset{\ref{elm:3}}{\sim}$ is by 
$T_{x_{2n-2}+ x_{2n-1}} T_{x_{2n} + x_{2n+1}}$ 
$T_{x_{2n-2}+ x_{2n-1} + x_{2n}+x_{2n+1}}$. 
Therefore, by the induction on $n$, we see 
$a \overset{\ref{elm:2},\ref{elm:3}}{\sim} x_1 + x_2$. 
Hence, $[a,b] \overset{\ref{elm:2},\ref{elm:3}}{\sim} [x_1+x_2, b']$. 
Since $0 = (x_1+x_2 \cdot b')_2$, 
$b' = x_1 + x_2 + x_l + \cdots + x_m$ or $b'= x_l + \cdots + x_m$ where $l \geq 3$. 
For these 2 cases, 
$T_{x_1+x_2} T_{b'} T_{x_1+x_2+b'} = T_{x_1+x_2} T_{x_1+x_2+b'} T_{b'}$ 
are the same. 
So, we may suppose $b'= x_l + \cdots + x_m$. 
We see $[a,b] \overset{\ref{elm:2},\ref{elm:3}}{\sim} 
[x_1 + x_2, (x_3 + x_4) + \cdots + (x_{2k-1} + x_{2k})]$, 
by applying (\ref{elm:2}) and (\ref{elm:3}) whose transvections 
are about $v$ such that $\text{supp}(v)$ contains neither $x_1$ nor $x_2$. 
If $2k \not= g$, by applying to $b$ the same argument as to $a$, 
we see $[a,b] \overset{\ref{elm:2}, \ref{elm:3}}{\sim} [x_1 + x_2, x_3 + x_4]$.
If $2k = g$, then $T_a \ T_b \ T_{a+b}$ is equal to 
$T_{x_1+x_2} T_{x_3 + \cdots + x_g} T_{x_1 + x_2 + x_3 + \cdots + x_g}$. 
In the last paragraph of this proof, we show that 
$T_{x_1+x_2} T_{x_3 + \cdots + x_g} T_{x_1 + x_2 + x_3 + \cdots + x_g}$ 
is a product of  (\ref{elm:2}) and (\ref{elm:3}). 

If $2n=g$, then $(v \cdot (x_1 + x_2) + \cdots +(x_{2n-1} + x_{2n}))_2 = 0$ 
for every $v \in H_1(N_g;\mathbb{Z}_2)$ used for the transvections in 
(\ref{elm:2}) and (\ref{elm:3}). 
Therefore, we see 
$[a,b] \overset{\ref{elm:2}, \ref{elm:3}}{\sim} 
[(x_1 + x_2) + \cdots + (x_{g-1} + x_g), 
(x_{2i+1} + x_{2i+2}) + \cdots + (x_{g-1} + x_g)]$. 
This means that $T_a \ T_b \ T_{a+b} = T_{a+b} \ T_b \ T_a$ is conjugate to 
$T_{(x_1 + x_2) + \cdots + (x_{2i-1} + x_{2i})}\  
T_{(x_{2i+1} + x_{2i+2}) + \cdots + (x_{g-1} + x_g)} \ 
T_{(x_1 + x_2) + \cdots + (x_{g-1} + x_g)}$ 
by (\ref{elm:2}) and (\ref{elm:3}). 
Moreover, 
$[((x_1 + x_2) + \cdots + (x_{2i-3} + x_{2i-2}) + (x_{2i-1} + x_{2i})), 
(x_{2i+1} + x_{2i+2}) + \cdots + (x_{g-1} + x_g)] 
\overset{\ref{elm:3}}{\sim} 
[(x_1 + x_2) + \cdots + x_{2i-3} + x_{2i}, 
x_{2i-2}+x_{2i-1}+(x_{2i+1}+x_{2i+2})+ \cdots + (x_{g-1} + x_g)] 
\overset{\ref{elm:2}}{\sim}
[(x_1 + x_2) + \cdots + (x_{2i-3} + x_{2i-2}), 
(x_{2i-1}+x_{2i})+(x_{2i+1}+x_{2i+2})+ \cdots + (x_{g-1} + x_g)]$, 
where $\overset{\ref{elm:3}}{\sim}$ is by 
$T_{x_{2i-2}+x_{2i-1}} T_{x_{2i}+x_{2i+1}} T_{x_{2i-2}+x_{2i-1}+x_{2i}+x_{2i+1}}$, 
and $\overset{\ref{elm:2}}{\sim}$ is by $T_{x_{2i-2}+x_{2i}}$. 
By repeatedly applying the above argument, we see that 
$T_a \ T_b \ T_{a+b}$ is conjugate to 
$T_{x_1+x_2} T_{x_3 + \cdots + x_g} T_{x_1 + x_2 + x_3 + \cdots + x_g}$ 
by (\ref{elm:2}) and (\ref{elm:3}). 

When $g$ is even and $g \geq 6$, by checking the action of transvections 
on the basis $\{ x_1, \ldots, x_g \}$ of $H_1(N_g ; \mathbb{Z}_2)$, 
we show $T_{x_1+x_2} T_{x_3 + \cdots + x_g} T_{x_1 + x_2 + x_3 + \cdots + x_g}$ 
$=$ $T_{x_1+x_2} T_{x_3+x_4}$ $T_{x_1+x_2+x_3+x_4} \cdot 
T_{x_1+x_2} T_{x_5 + \cdots + x_g} T_{x_1+x_2+x_5 + \cdots + x_g}$. 
Since $[x_1+x_2, x_5 + \cdots + x_g] \overset{\ref{elm:2}}{\sim} 
[x_1+x_2, x_3 + \cdots + x_{g-2}]$, 
we see that $T_{x_1+x_2} T_{x_5 + \cdots + x_g} T_{x_1+x_2+x_5 + \cdots + x_g}$
is a product of (\ref{elm:2}) and (\ref{elm:3}) by using the argument 
for the case where $2n(=g-2)\not= g$. 
\end{proof}
Since $(t_{d_i})_* = T_{x_i+x_{i+2}}$,  
$(t_{a_i} t_{a_{i+2}} t_{c_i})_*$ $=$ 
$T_{x_i+x_{i+1}} T_{x_{i+2}+x_{i+3}} T_{x_i+x_{i+1} + x_{i+2}+x_{i+3}}$, 
and $t_{d_i}$, $t_{a_i} t_{a_{i+2}} t_{c_i}$ $\in G_g$, 
we see $G_g = \mathcal{N}_g(q_{os})$, hence Theorem \ref{thm:generator-pin} follows.  
\subsection*{Acknowledgments}
The author wishes to express his gratitude to Professors Akio Kawauchi and 
Masamichi Takase for advising the author for considering on $o$-standard embeddings 
and Professor B{\l}a\.{z}ej Szepietowski for informing and sending 
the author his preprint \cite{Szepietowski2}. 

\end{document}